\documentclass{amsart}
\usepackage{graphicx,amsfonts,amssymb,xcolor} 

\title[Non-solutions from random walks]{Non-solutions to mixed equations in acylindrically hyperbolic groups coming from random walks }
\author{Henry Bradford}
\address{Christ's College, University of Cambridge; St Andrew's Street; Cambridge CB2 3BU; ENGLAND}
    \email{hb470@cam.ac.uk}

\author{Alessandro Sisto}
\address{Department of Mathematics, Heriot-Watt University and Maxwell Institute for Mathematical Sciences, Edinburgh, UK}
    \email{a.sisto@hw.ac.uk}

\newtheorem{thm}{Theorem}[section]
\newtheorem{lem}[thm]{Lemma}

\newtheorem{coroll}[thm]{Corollary}
\newtheorem{defn}[thm]{Definition}

\newtheorem{rmrk}[thm]{Remark}
\newtheorem{qu}[thm]{Question}

\DeclareMathOperator{\id}{id}
\DeclareMathOperator{\PSL}{PSL}

\begin{document}

\begin{abstract}
A mixed equation in a group $G$ is given by a non-trivial element $w (x)$ of the free product $G \ast \mathbb{Z}$, and a solution is some $g\in G$ such that $w(g)$ is the identity. For $G$ acylindrically hyperbolic with trivial finite radical (e.g. torsion-free) we show that any mixed equation of length $n$ has a non-solution of length comparable to $\log(n)$, which is the best possible bound. Similarly, we show that there is a common non-solution of length $O(n)$ to all mixed equations of length $n$, again the best possible bound. In fact, in both cases we show that a random walk of appropriate length yields a non-solution with positive probability.
\end{abstract}

\maketitle

\section{Introduction}

A group is mixed-identity free (MIF) if for any non-trivial element element $w(x) \in G \ast \langle x \rangle$ in $G$ there exists $g\in G$ such that $w(g)$ is not the identity. 
Here and throughout, $x$ is a free variable, generating an infinite cyclic group. 

For a MIF group, it is natural and useful to quantify the length of the shortest possible non-solutions to mixed identities; this has been done for instance in \cite{Brad,BradSchnThom,GrovesWilton,AmGaElPa,Vigdo}, most recently motivated by questions in the study of operator algebras.

There are two ways to quantify the length of the shortest possible non-solutions to mixed identities, discussed in the following two subsections. In both cases we show that random walks yield non-solutions of optimal length.

\subsection{MIF growth} Let $G$ be MIF, generated by a finite set $S$. In this paper, given a word metric on a group $G$, we automatically fix a word metric on $G\ast\mathbb{Z}$ by augmenting our generating set for $G$ by our fixed generator $x$ for $\mathbb{Z}$. The \emph{complexity} of an element $w(x) \in G \ast \langle x \rangle$ in $G$ 
is the minimal word-length in $G$ of an element $g \in G$ satisfying $w(g) \neq e_G$. 
The \emph{MIF growth} function $\mathcal{M}_G ^S : \mathbb{N} \rightarrow \mathbb{N}$ 
of $G$ is defined such that $\mathcal{M}_G ^S (n)$ is the maximal complexity 
of an element of $G \ast \langle x \rangle$ of word length at most $n$ 
(with respect to $S \cup \lbrace x \rbrace$). 
For every finitely generated MIF group $G$, $\mathcal{M}_G$ grows at least logarithmically (see \cite{Brad} Theorem 1.9), and following \cite{BradSchnThom} we call $G$ \emph{sharply MIF} if $\mathcal{M}_G$ 
grows \emph{at most} logarithmically. 
Our first main result is as follows, and we note that the theorem is new even for non-elementary hyperbolic groups.

\begin{thm} \label{MainThmIntro}
Every finitely generated acylindrically hyperbolic group with trivial finite radical is sharply MIF. 
\end{thm}

A group is \emph{acylindrically hyperbolic} if it admits a nonelementary  acylindrical action on a hyperbolic metric space; this is a huge class of groups and we refer the reader to \cite{Osin:ICM} for a survey. 
The fact that acylindrically hyperbolic group with trivial finite radical are MIF is due 
to Hull and Osin (\cite{HullOsin} Corollary 1.7). 

The key novelty of this paper is the use of random walks to find non-solutions to mixed identities, and indeed Theorem \ref{MainThmIntro} is an immediate consequence of the next result about random walks. 

We call a measure on an acylindrically hyperbolic group \emph{admissible} if it is finitely supported and the semigroup generated by the support is the whole of $G$.

\begin{thm}
\label{thm:random}
    Let $G$ be a finitely generated acylindrically hyperbolic group with trivial finite radical, and let $(x_n)$ be a random walk driven by an admissible measure. Then there exists a constant $C>0$ such that for any $n>1$ and non-trivial $w(x) \in G \ast \langle x \rangle$ of word length at most $n$ we have
    $$\mathbb P[w(x_{\lceil C\log n\rceil})\neq e_G]>0.$$  
\end{thm}

Our Theorem \ref{MainThmIntro} is a wide generalization of previous work of the 
first author, Schneider and Thom, who showed that finitely generated torsion-free nonelementary Kleinian group are sharply MIF \cite[Theorem 1.2]{BradSchnThom}. The proof was by completely different methods from our work here, relying as it did on expansion properties of Cayley graphs for finite simple groups, 
and some elementary algebraic geometry. 
The methods used by Hull and Osin in \cite{HullOsin} are different again. 


\subsection{Simultaneous non-solutions}
Our methods also allow us to show that outcomes of sufficiently long random walks 
are, with high probability, \emph{simultaneously} non-solutions 
to large families of mixed identities. The study of lengths of non-solutions to families of mixed identities originates in \cite{AmGaElPa} and was continued in \cite{Vigdo}, motivated by problems in operator algebras.

Given a word metric on the group $G$ and $n \geq 1$, 
let $W_n \subseteq G\ast\mathbb{Z}$ be the set of 
all nontrivial elements of $G\ast\mathbb{Z}$, all of whose 
constants have word length at most $n$ in $G$ in $G\ast\mathbb{Z}$. 

\begin{thm}
\label{thm:linear_MIF}
Let $G$ be a finitely generated acylindrically hyperbolic group 
with trivial finite radical, and fix a word metric on $G$. Also, let $(x_n)$ be a random walk driven by an admissible measure.
Then there exists an integer $C\geq 1$ such that 
for all integers $n\geq 1$, with positive probability we have $w(x_{Cn})\neq e_G$ for all $w(x)\in W_n$. 
In particular $w(x_{Cn})\neq e_G$ for all non-trivial $w(x)\in G\ast\mathbb{Z}$ of word length at most $n$.
\end{thm}

Being sharply MIF implies that there exist simultaneous non-solutions of length $O(n)$ to all mixed equations of length $n$; this can be shown using iterated commutators somewhat similarly to \cite[Proposition 5.3]{HullOsin}. We thank Srivatsav Kunnawalkam Elayavalli for this observation. Using Theorem \ref{thm:linear_MIF} we get even a little bit more:

\begin{coroll}
\label{coroll:linear_MIF}
Under the conditions of Theorem \ref{thm:linear_MIF}, 
for all integers $n\geq 1$, 
there exists $g_n \in G$ with $\lvert g_n \rvert = O(n)$ 
such that $w(g_n)\neq e_G$ for all $w(x)\in W_n$. 
\end{coroll}


 In \cite{AmGaElPa} the authors obtain an estimate of $O(n^{12})$ rather than $O(n)$ for a simultaneous non-solution to all mixed equations of length at most $n$.
The conclusion of Corollary \ref{coroll:linear_MIF} (with an estimate $O(n)$) 
is also known to hold for 
cocompact lattices in $\PSL_d$ over local fields for $d \geq 2$ (Theorem 1.2 of \cite{Vigdo}). This will be extended to all linear groups in \cite{AG:effective}.
The estimate $O(n)$, for any MIF group, is easily seen to be optimal, 
by considering the commutator words $[x,g]$, for $\lvert g \rvert \leq n$. 
We also note that the estimates in Theorem \ref{MainThmIntro} and 
Corollary \ref{coroll:linear_MIF} would remain unchanged if we considered 
instead word equations in $G \ast \mathbb{F}$, for any free group $\mathbb{F}$ of finite rank $r \geq 1$ (with any word metric on non-solutions in $G^r$), 
by standard embeddings between $G \ast \mathbb{F}$ and $G \ast \mathbb{Z}$. 

Theorem \ref{thm:linear_MIF} is relevant to another effective version of MIF, again introduced in \cite{AmGaElPa}, and called selfless-ness. We recall this notion and make some remarks about this in Section \ref{sec:self}.

\subsection*{Acknowledgments} We would like to thanks Srivatsav Kunnawalkam Elayavalli, Francesco Fournier Facio, and Tsachik Gelander for useful discussions, and Giorgio Mangioni for pointing out the reference \cite{MT:Cremona}.


\section{Proofs}

We will use a concatenation lemma to certify non-triviality of certain elements of groups acting on hyperbolic spaces. There are several variations on this in the literature, the one that works for us is the following, from \cite{HS:separable}. Roughly, it says that given a concatenation of geodesics in a hyperbolic space, for the concatenation to be a quasi-geodesic it suffices for the even geodesics to not overlap much with the two geodesics coming right after and the two geodesics coming right before.

\begin{lem}\label{lem:concat}\cite[Lemma 2.3]{HS:separable}
For every $\delta\geq 0$ there exists a constant $C_\delta$ such that the following holds. Let $(\alpha_i)_{i\in \mathbb Z}$ be geodesics in a $\delta$-hyperbolic geodesic space, of length bounded independently of $i$, with the terminal point of $\alpha_i$ being the initial point of $\alpha_{i+1}$. Let $$d_{i,j}=diam (N_{3\delta}(\alpha_i)\cap \alpha_j)$$ and assume that for all even $i$, 
$$\ell(\alpha_i)\geq d_{i,i+1}+d_{i,i+2}+d_{i,i-1}+d_{i,i-2}+C_\delta.$$
Then the concatenation of the $\alpha_i$ is a quasi-geodesic.
\end{lem}

Fix a group $G$ acting acylindrically and nonelementarily on the hyperbolic space $X$, and a random walk $(x_n)$ driven by an admissible measure on $G$. For convenience, we can and will assume that $X$ is a Cayley graph of $G$, so that group elements are also points in $X$ 
(see Theorem 1.2 of \cite{Osin}). Fix a hyperbolicity constant $\delta$ for $X$.

First of all, we recall the fundamental result that $(x_n)$ makes linear progress in the $X$:

\begin{thm}\cite{MaherTiozzo}
\label{thm:lin_prog}
There exists $\lambda>0$ such that
$$\mathbb P[d_X(e,x_n)> \lambda n]\to 1.$$
\end{thm}

For any $g,h\in G$ fix a geodesic $[g,h]^X$ in $X$ connecting $g,h$ (seen as elements of $X$).

For $g,h\in G$ denote
$$\mathcal D(g,h)=diam\left(N_{3\delta}([e,g]^X)\cap [e,h]^X\right).$$

This is essentially a Gromov product, and the following lemma says that the Gromov product of a random element with any fixed element is small with high probability.

\begin{lem}
\label{lem:gr_prod}
    There exists a constant $C_1>0$ such that for all $g\in G$ and $t\geq 0$ and positive integer $n$ we have
    $$\mathbb P[\mathcal D(g,x_n)\geq t]\leq C_1 2^{-t/C_1},$$
    and similarly for $x_n^{-1}$.
\end{lem}

\begin{proof}
 For $x_n$, this follows from \cite[Lemma 2.10]{Maher:exp} (whose proof, as observed in \cite[Proposition 5.1]{MaherTiozzo}, applies to our setting). For $x_n^{-1}$ simply observe that the law of $x_n^{-1}$ is the law of the random walk driven by the reflected measure.
\end{proof}

We need one more lemma before proving Theorem \ref{thm:random}. Roughly, this says that a random geodesic does not overlap much with any of its translates.

\begin{lem}
\label{lem:match}
Suppose that $G$ has trivial finite radical. Then for all $\epsilon>0$ we have
    $$\mathbb P[\forall g:\ \mathcal D(x^{-1}_n,gx_n)\leq \epsilon n]\to 1$$
    and, similarly,
    $$\mathbb P[\forall g\neq e:\ \mathcal D(x_n,gx_n)\leq \epsilon n]\to 1.$$
\end{lem}

\begin{proof}
This is essentially \cite[Proposition 11.8]{MT:Cremona} combined with \cite[Proposition 5.1]{MS:rnd_subgp}. The former is stated in terms of the axes of $x_n$ and a conjugate, but the first paragraph of the proof actually reduces to two translates of a geodesic to the random element.
\end{proof}

\subsection{Proof of Theorem \ref{thm:random}}

Fix $\lambda$ as in Theorem \ref{thm:lin_prog} and $C_1$ as in Lemma \ref{lem:gr_prod}. Set $C=20C_1/\lambda$.

Let $w(x) \in G \ast \langle x \rangle$ be of word length $n>1$, which we can take to be cyclically reduced and we can assume $n$ to be sufficiently large. Let $\{g_i\}_{i=1,\dots,k}$ be the collection of all elements of $G$ appearing in $w(x)$ together with their inverses and the identity, and notice $k\leq 2n$. Set $m=\lceil  C\log_2 n\rceil$.

By Lemma \ref{lem:gr_prod} we have
\begin{equation} \label{DivergeConstsEqn}
    \mathbb P[\exists g_i:\ \mathcal D(g_i,x_m)\geq \lambda m/10]\leq \sum_{i=1}^k C_1 2^{-C\lambda \log_2(n)/(10C_1)}\leq 2nC_1 n^{-2}=2C_1/n.
\end{equation}
Hence, and in view of Theorem \ref{thm:lin_prog} and Lemma \ref{lem:match}, for sufficiently large $n$ we have that with positive probability all of the following hold:
\begin{itemize}
 \item $\mathcal D(g_i,x_m)\leq \lambda m/10$ for all $i$ (by (\ref{DivergeConstsEqn})),
 \item $d_X(e,x_m)> \lambda m$,
 \item $\mathcal D(x^{-1}_m,g_ix_m)\leq \lambda m/10$ for all $i$,
 \item $\mathcal D(x_m,gx_m)\leq \lambda m/10$ for all $i$ with $g_i\neq e$.
\end{itemize}

We can now form a bi-infinite concatenation of geodesics in $X$, by concatenating translates of geodesics of the form $[1,g_i]^X$ and $[1,x^{\pm 1}_m]^X$, according to a bi-infinite sequence repeating $w(x)$. For convenience, between any two translates of $[1,x^{\pm 1}_m]^X$ we artificially insert a trivial geodesic in the concatenation, and label the various geodesics in a way that the translates of $[1,x^{\pm 1}_m]^X$ are the even ones. Then this concatenation satisfies the conditions of Lemma \ref{lem:concat}, provided that $n$ is large enough, and therefore the bi-infinite sequence $\big(w(x_n)^j \big)_{j \in \mathbb{Z}}$ follows a quasi-geodesic in
$X$. In particular $w(x_n)$ has infinite order in $G$, so is nontrivial, as required.\qed

\subsection{Proof of Theorem \ref{thm:linear_MIF}} The proof is very similar. Again we $\lambda$ as in Theorem \ref{thm:lin_prog} and $C_1$ as in Lemma \ref{lem:gr_prod}. We also let $s$ be twice the cardinality of the generating set of $G$ giving the word metric under consideration. Set $C=20C_1\log_2(s)/\lambda$. 

Given an integer $n\geq 1$, we now let $\{g_i\}_{i=1,\dots,k}$ be the collection of all non-trivial elements in the ball of radius $n$ in $G$, and we note $k\leq s^n$. Setting $m=Cn$, by Lemma \ref{lem:gr_prod} we have

$$\mathbb P[\exists g_i:\ \mathcal D(g_i,x_{m})\geq \lambda m/10]\leq \sum_{i=1}^k C_1 2^{-\lambda C n/(10C_1)}\leq C_1 s^n s^{-2n}=C_1s^{-n}.$$

Given this, for $n$ sufficiently large we have that with positive probability direct analogues of the four bullet points in the proof of Theorem \ref{thm:random} hold for $x_m$ and any non-trivial $w(x)\in G\ast \mathbb Z$. The same argument then shows that with positive probability we have $w(x_m)\neq e_G$ for all such $w(x)$, as required.\qed

\section{Selfless-ness}
\label{sec:self}

In this section we discuss selfless-ness, as introduced in \cite{AmGaElPa}, which is a way of quantifying being MIF.

\begin{defn}
\label{defn:selfless}
For $f : \mathbb{N} \rightarrow \mathbb{R}$
a function satisfying:
\begin{equation} \label{SelflessLimit}
\liminf_n f(n)^{\frac{1}{n}} = 1
\end{equation}
an infinite group $G$ generated by a finite set $S$
is \emph{$f$-selfless} if, for all $n \in \mathbb{N}$,
there exists a homomorphism
$\phi_n : G \ast \langle x \rangle \rightarrow G$ such that:
\begin{itemize}
\item[(i)] $\phi_n |_G = \id_G$;

\item[(ii)] $\phi_n |_{B_{S \cup \lbrace x \rbrace}(n)}$ is injective;

\item[(iii)] $\phi_n (B_{S \cup \lbrace x \rbrace}(n))
\subseteq B_S (f(n))$.

\end{itemize}
The group $G$ is \emph{selfless} if it is
$f$-selfless for some function $f$ as above.
\end{defn}

\begin{rmrk}
\label{rmrk:selfless}
\begin{itemize}
\item[(i)] Every selfless group is MIF. Moreover, 
if $G$ is $f$-selfless, 
then $\mathcal{M}_{G} ^S (n)\leq f(n)$ for all $n\in\mathbb{N}$. 

\item[(ii)] Since for all $n \geq 1$, 
\begin{equation*}
\lvert B_S (n-1) \rvert \lneq \lvert B_S (n) \rvert = \lvert \phi_n \big(B_S (n)\big) \rvert
\leq \lvert B_S \big(f(n)\big) \rvert , 
\end{equation*}
if $G$ is $f$-selfless, then $f(n) \geq n$ 
for all $n \in \mathbb{N}$. 
In particular the inequality $\mathcal{M}_G ^S (n)\leq f(n)$ 
above fails to be sharp for any sharply MIF group. 

\item[(iii)] Suppose for each $n \geq 1$ we are given $g_n \in G$ satisfying 
$\liminf_n \lvert g_n \rvert^{\frac{1}{n}} = 1$ and $w(g_n) \neq e_G$ 
for all $w(x) \in B_{S \cup \lbrace x \rbrace}(n)$. 
Then setting $\phi_n (x) = g_{2n}$, the conditions of Definition \ref{defn:selfless} 
are satisfied, with $f(n) = n \lvert g_{2n} \rvert$. 
\end{itemize}
\end{rmrk}

Although it is clear that every finitely generated MIF group has finite 
MIF growth, to our knowledge, the following is open. 

\begin{qu}
Is every finitely generated MIF group selfless? 
\end{qu}

It is proved in \cite[Theorem 3.3]{AmGaElPa} that acylindrically hyperbolic groups with trivial finite radical are selfless. The proof in fact yields that every such $G$ 
satisfies the conclusion of Corollary \ref{coroll:linear_MIF}, 
but with the weaker bound $\lvert g_n \rvert = O(n^{12})$. 
Hence, by Remark \ref{rmrk:selfless} (iii), 
\cite[Theorem 3.3]{AmGaElPa} yields that every such $G$ 
is $f$-selfless, for some $f(n)=O(n^{13})$, 
and our Corollary \ref{coroll:linear_MIF}, or being sharply MIF, gives $f(n)=O(n^{2})$:

\begin{coroll}
    Let $G$ be a finitely generated acylindrically hyperbolic group 
with trivial finite radical. Then $G$ is $f$-selfless for some $f=O(n^2)$.
\end{coroll}

However, we do not know the answer to the following:

\begin{qu}
    Are there groups which are $f$-selfless for some $f=o(n^2)$?
\end{qu}

Note that the image of $x$ under $\phi_n$ as in Definition \ref{defn:selfless} needs to have length at least $O(n)$, but $\phi_n(x^n)$ could then have word length much shorter than $n|\phi_n(x)|$. This prevents us from concluding that the selfless-ness function cannot be $o(n^2)$, and therefore we do not know whether our estimate of $O(n^2)$ is optimal for all acylindrically hyperbolic groups.










\bibliographystyle{alpha}
\bibliography{biblio}

%
%
%
%

\end{document}